\documentclass[12pt, reqno]{amsart}

\usepackage{amsfonts,amsmath}
\usepackage{amscd}
\usepackage{amssymb}

\allowdisplaybreaks

\usepackage{hyperref}

\date{\today}

\newtheorem{thm}{Theorem}[section]
\newtheorem{cor}[thm]{Corollary}

\newtheorem{prop}[thm]{Proposition}
\theoremstyle{definition}

\theoremstyle{remark}
\newtheorem{rem}[thm]{Remark}

\numberwithin{equation}{section}

\newcommand{\R}{\mathbb R}

\newcommand{\He}{\mathbb H}

\newcommand{\C}{{\mathbb C}}

\renewcommand{\Im}{\operatorname{Im}}
\newcommand{\tr}{\operatorname{tr}}

\usepackage{color}

\title[Beurling's theorem on the Heisenberg group]
{  Beurling's theorem on the Heisenberg group}

\author[ S. Thangavelu]{  Sundaram Thangavelu}

\address[S. Thangavelu]{Department of Mathematics\\
Indian Institute of Science\\
560 012 Bangalore, India}

\email{veluma@iisc.ac.in}

\date{}
 
\begin{document}

\maketitle

\begin{abstract}  We formulate and prove an analogue of Beurling's theorem for the Fourier transform on the Heisenberg group. As a consequence we deduce Hardy and Cowling-Price  theorems.
 \end{abstract}

\section {Introduction}
There are several uncertainty principles for the Fourier transform on $ \R^n $ such as theorems of Hardy, Cowling and Price, Gelfand and Shilov and Beurling.  Each of the  first three of these theorems can be deduced as corollaries of the fourth one which we recall now. Beurling's theorem states that there is no nontrivial function which satisfies
\begin{equation}\label{B-C} \int_{\R^n} \int_{\R^n} |f(y)| \,   |\widehat{f}(\xi)| \, e^{|(y, \xi)|} \, dy d\xi   < \infty .
\end{equation}
A proof of this theorem in the one dimensional case was published by H\"ormander in \cite{LH}.  Later Bonami et al \cite{BDJ}  extended the result to higher dimensions.  Recently, Hedenmalm \cite{HH} came up with a very elegant and simple proof of  Beurling's theorem, which we will recall later for the convenience of the reader, see subsection 3.1.\\

 Observe that the Beurling's condition (\ref{B-C}) can be restated  in the following form: 
 $$   w_0(\widehat{f},y) =: \int_{\R^n}  |\widehat{f}(\xi)| e^{|(y, \xi)|} d\xi, \,\,\,  \int_{\R^n} |f(y)| w_0(\widehat{f}, y) dy < \infty.$$
 The function $ w_0(\hat{f},y) $ is an increasing function of $ |y| $ and  dominates $ |f(x+iy)|:$
 $$ |f(x+iy)| \leq (2\pi)^{-n/2}  \int_{\R^n}  |\widehat{f}(\xi)| e^{|(y, \xi)|} d\xi \leq C\,  w_0(\widehat{f},y).$$
 Here we have tacitly used the fact that under the Beurling's condition, $ f $ extends to $ \C^n $ as a holomorphic function. 
Once the holomorphic extendability of $ f $ is given, Hedenmalm's proof  goes through under the  modified Beurling's condition 
\begin{equation}\label{mod-b}
 \sup_{x\in \R^n} |f(x+iy)| \leq C w(\widehat{f}, y) ,\,\,\,    \int_{\R^n} |f(y)| w(\widehat{f}, y) dy < \infty, 
  \end{equation}
for some  $ w(\widehat{f}, y) $ which is increasing as a function of $ |y|.$   In particular, we can take
$$ w_1(\widehat{f},y) = \sup_{|u| \leq |y|} \int_{\R^n}  |\widehat{f}(\xi)| \, e^{- u \cdot \xi} d\xi.  $$ 
Thus Beurling's theorem can be stated in the following form.

\begin{thm}\label{B-E} 
There is no nontrivial function  $ f $ on $ \R^n $ for which
\begin{equation} \int_{\R^n} |f(y)|  \Big( \sup_{|u| \leq |y|} \int_{\R^n}  |\widehat{f}(\xi)| \, e^{- u \cdot \xi} d\xi \Big)  dy  < \infty .
\end{equation}
\end{thm}

In this article our aim is to prove an analogue of Theorem \ref{B-E} for the Fourier transform on the Heisenberg group. In order to state our result, let us recall some definitions and set the stage. Given a function  $ f $ on the Heisenberg group $ \He^n $ its Fourier transform $ \hat{f}(\lambda) $ is the operator valued function on $ \R^\ast $  defined by
$$  \hat{f}(\lambda) = \int_{\He^n} f(z,t) \, \pi_\lambda(z,t) \, dz dt $$
where $ \pi_\lambda(z,t) = e^{i\lambda t} \pi_\lambda(z) $ are the Schr\"odinger representations, all realised on the same Hilbert space $ L^2(\R^n).$  Instead of $ (z,t) $ we will use the real coordinates $ (x,u,t) $ on $ \He^n $ which allows us to complexify $ x $ and $ u $ into $ z= x+iy, w = u+iv.$ The operators  $ \pi_\lambda(x,u) $ have  natural extensions to complex values as unbounded operators $ \pi_\lambda(z,w).$  We  let $ f^\lambda(x,u) $ stand for the inverse Fourier transform of $ f $ in the  $t$-variable so that
$$ \widehat{f}(\lambda) = \int_{\R^{2n}} f^\lambda(x,u) \, \pi_\lambda(x,u) \, dx du = \pi_\lambda(f^\lambda).$$
We need to make a preliminary assumption on $ \widehat{f}(\lambda) $ which allows us to holomorphically extend $ f^\lambda(x,u) $ to $ \C^{2n}.$\\

We consider the action of $ U(n) $ on $ \He^n$ given by $ \sigma(z,t) = (\sigma z, t) $ which is an automorphism of $ \He^n$. Let $ \sigma(x,u) $ stand for the action of $ U(n) $ identified with $ Sp(n,\R) \cap O(2n,\R) $ on $ \R^{2n}.$ It has a natural action on $ \C^n \times \C^n $ and so it makes sense to talk about $ \sigma (z, w) $ for $ z, w \in \C^n.$  On the Fourier transform of $ \hat{f}(\lambda) $ we impose the condition
\begin{equation}\label{condi-one}
  \int_{U(n)}  \|\pi_\lambda(\sigma(z,w)) \widehat{f}(\lambda)\|_{HS}^2  \, d\sigma  < \infty.
  \end{equation}
Using Gutzmer's formula for special Hermite expansions, we will show that (\ref{condi-one}) allows us to extend $ f^\lambda(x,u) $ holomorphically.  We then look for a radial function $ F_\lambda(y,v) $ satisfying
\begin{equation}\label{condi-two} | \tr \big( \pi_\lambda(x+iy,u+iv)^\ast \widehat{f}(\lambda)\big) | \leq C_\lambda e^{\frac{1}{2}\lambda(u\cdot y- v\cdot x)}\, F_\lambda(y,v) 
\end{equation}
where $ C_\lambda $ is independent of $ (x,u).$ In view of the relation  (as $ \pi_\lambda(iy,iv) $ are self-adjoint) 
$$  \pi_\lambda(iy,iv) \pi_\lambda(x,u)  = \pi_\lambda(x+iy, y+iv) e^{-\frac{1}{2}\lambda(u\cdot y- v\cdot x)} $$ 
a candidate for  $ F_\lambda(y,v) $ is given by  the following analogue of $ w_1(\hat{f},y) $ on $ \R^n$ defined earlier:
\begin{equation}\label{candi}  w_\lambda(\widehat{f},(y,v)) =   \sup_{|(y^\prime,v^\prime) \leq |(y,v)| } \| \pi_\lambda((i(y^\prime,v^\prime))  \widehat{f}(\lambda) \|_1  \end{equation}
where $ \|T\|_1 = \tr |T| ,\, |T| = (T^\ast T)^{1/2}$ is the trace norm of the operator. Having set up the notations, here is our Beurling's theorem.

\begin{thm}\label{B-H}  Assume that  $ \mathcal{R}(\widehat{f}(\lambda)) \subset \mathcal{D}(\pi_\lambda(z,w))$ and $ \pi_\lambda(z,w) \widehat{f}(\lambda) $ is of trace class for all $ (z,w) \in \C^{2n} $   such that
the assumption (\ref{condi-one}) holds. Then the following  condition cannot hold for all $ \lambda$ unless $ f =0 :$ 
\begin{equation}\label{mod-b-h}     \int_{\R^{2n}}  | f^{\lambda}(y,v)|\,\, \Big(  \sup_{|(y^\prime,v^\prime)| \leq |(y,v)| } \| \pi_\lambda((i(y^\prime,v^\prime))  \widehat{f}(\lambda) \|_1 \Big)  dy\,dv  < \infty.
\end{equation}

\end{thm} 

\begin{rem} An examination of the proof will show that in the above theorem  we can replace $ w_\lambda(\widehat{f},(y,v))$ by any radial increasing $ F_\lambda(y,v) $ satisfying (\ref{condi-two}).
Thus the conclusion of the Theorem \ref{B-H} holds under the assumption that 
\begin{equation}\label{mod-b-h-1}     \int_{\R^{2n}}  | f^{\lambda}(y,v)|\,\, F_\lambda(y,v)   dy\,dv  < \infty.
\end{equation}
It is this flexibility with the choice of $ F_\lambda(y,v) $ that allows  us to deduce  theorems of Hardy and Cowling-Price from Theorem \ref{B-H}.\\
\end{rem}

As we have mentioned in the beginning, many uncertainty principles on $ \R^n $ such as theorems of Hardy, Cowling-Price and Gelfand-Shilov can be easily deduced from Beurling's theorem. In the case of Hardy conditions, namely
$$ |f(x)| \leq c_1\, q_a(x),\,\,\, |\widehat{f}(\xi)| \leq c_2\, \widehat{q_b}(\xi) ,$$ 
stated in terms of the heat kernel $ q_t $ associated to the Laplacian,  Fourier inversion gives us
$$ |f(x+iy)| \leq (2\pi)^{-n/2} \int_{\R^n} |\widehat{f}(\xi)|\, e^{- y\cdot \xi} d\xi \leq C\, \int_{\R^n} e^{-b|\xi|^2} \, e^{- y \cdot \xi} d\xi = C^\prime q_b(iy) .$$
Thus when $ a < b $ the condition (\ref{mod-b}) is satisfied with $ w(\hat{f},y) = q_b(iy)$ and we obtain Hardy's theorem. In the case of Cowling-Price theorem we assume
$$ f(x) = g(x) q_a(x),\, \widehat{f}(\xi) = h(\xi) \widehat{q_b}(\xi),\,\,  g, h \in L^2(\R^n).$$ 
Consequently, $ f = \varphi \ast q_b $ where $ \widehat{\varphi} = h,$ and hence $ f $ belongs to the weighted Bergman space $ \mathcal{B}_b(\C^n) $ consisting of entire functions which are square integrable on $ \C^n$ with respect to the weight function $ q_{2b}(2y).$ The reproducing kernel for this Hilbert space is given by $ K(z,w) = q_{2b}(z-\bar{w}) $ and from general theory we get the estimate
$$  |f(x+iy)|^2 \leq C\, K(z, z) = C\, q_{2b}(2iy) = C^\prime\, (q_b(iy))^2.$$
Thus we can again take  $ w(\widehat{f},y) = q_b(iy)$ in  (\ref{mod-b}) and Cowling-Price theorem follows.\\

More generally, theorems of Hardy and Cowing-Price can be put in the general frame work of   weighted Bergman spaces $ \mathcal{B}(\C^n, w) $ which  are invariant under the natural action of  $ \R^n$ by translations. This invariance forces $ w $ to be a function of $ y $ alone and the reproducing kernel takes the form $ K(z,w) = k(z-\bar{w}) $ for some function $ k.$  We are interested in the situation when $ \mathcal{B}(\C^n, w) $ occurs as the image of  $ L^2(\R^n) $ under a transform $ T_qf(z) = f \ast q(z) $ for some nice function $ q \in L^1(\R^n).$  Under these circumstances we have 
$$ \int_{\C^n} |f \ast q(z)|^2 \, w(y) \, dz = c_n \int_{\R^n} |f(x)|^2 \, dx.$$
Moreover, the reproducing kernel is given by $ K(z,w) = q \ast q(z-\bar{w}) $ and we can prove the following theorem.

\begin{thm}\label{B-E-1}  Let  $ q \in L^1(\R^n) $ be such that the image of $ L^2(\R^n) $ under the transform $  T_qf(z) =  f \ast q(z) $ is a weighted Bergman space $ \mathcal{B}(\C^n, w) $  where $ w$ is radial. Then there is no nontrivial $ f \in L^2(\He^n) $ for which 
$$ \int_{\R^n} | f \ast q(y)| \, \sqrt{q \ast q(2iy)} \, dy < \infty .$$
\end{thm} 

As observed above, both Hardy and Cowling-Price theorems correspond to the transform $ T_bf(z) = f \ast q_b(z) $ known as the Segal-Bargmann transform or heat kernel transform.\\

The relation between the weight function $ w $ and the reproducing kernel $ k $ can be read out by considering the orbital integrals of $ F $. Consider the action action of the Euclidean motion group $ M(n) $ on $ \R^n$ given by $ g .\, y = (x+\sigma y) $ for $ g = (x,\sigma) \in M(n), \sigma \in SO(n).$ This action has a natural  extension to $ \C^n$ simply given by $ g. (u+iv) = g. \,u +i g .\,v.$ The orbital integrals of $ |F|^2 $ are then defined by
$$ O_{|F|^2}(z) = \int_{M(n)} |F(g.\,z)|^2 \, dg.$$ 
As it turns out, $ O_{|F|^2}(x+iy) $ depends only on $ y $ which is radial and given explicitly in terms of $ \widehat{f}$ where $ f $ is just the restriction of $ F $ to $ \R^n.$
$$ O_{|F|^2}(z) = c_n\, \int_{\R^n} |\widehat{f}(\xi)|^2 \,  \frac{J_{n/2-1}(2i |y|  |\xi|)}{(2i|y||\xi|)^{n/2-1}} \, d\xi $$
where $ J_\alpha $ stands for the Bessel function of order $ \alpha.$ When the weight function $ w(y) $ is radial, its Fourier transform is given by Hankel transform and hence  we have
$$  \int_{\C^n} |F(z)|^2 \, w(y) dz  = \int_{\R^n}  O_{|F|^2}(y) \, w(y) \,dy = c_n \int_{\R^n} |\widehat{f}(\xi)|^2 \,  \widehat{w}(2i\xi) \, d\xi .$$
By defining $ q $ by the realtion $ \widehat{q}(\xi) =( w(2i\xi))^{-1/2} $ we can factorise $ f $ as $ f = h \ast q $  and from the above relation we get
$$ c_n \int_{\R^n} |\widehat{h}(\xi)|^2  \, d\xi  =  \int_{\C^n} |F(z)|^2 \, w(y) dz .$$
Thus we see that $ \mathcal{B}(\C^n, w) $ can be identified with the image of $ L^2(\R^n) $ under the transform $ h(x) \rightarrow h \ast q(z).$  We also note that the relation between $ q $ and $ w$ allows to calculate $ q\ast q $ in terms of $ w $ as follows:
$$ q\ast q(2iy) = c_n \int_{\R^n}  e^{2 y \cdot \xi} \, \widehat{w}(2i\xi)^{-1} \, d\xi.$$
From this, it is clear that $ q \ast q(2iy) $ is a radial increasing function.
The above arguments are formal  but we can make them rigorous by imposing certain  assumptions either on $ w $ or on $ q.$ \\

We can prove an analogue of Theorem \ref{B-E-1} for the Heisenberg group, the formulation of which requires the setting of twisted Bergman spaces. Given a positive radial weight function $ w_\lambda(y,v) $ on $ \R^{2n} $ we consider the twisted Bergman  space $ \widetilde{\mathcal{B}}_\lambda(\C^{2n}, w_\lambda) $ consisting of entire functions $ F(z,w)$ on $ \C^{2n} $ such that
$$   \| F\|^2 = \int_{\C^{2n}} |F(z,w)|^2 w_\lambda(y,v) \, e^{\lambda (u \cdot y- v \cdot x)} \, dz dw < \infty.$$
Such spaces are invariant under twisted translations and we restrict ourselves to the  situation where they can be realised as the image of $ L^2(\R^{2n}) $ under the map $ \widetilde{T}_{q_\lambda}f(z,w) =  f \ast_\lambda q_\lambda(z,w) $ for a nice kernel  function $ q_\lambda(x,u)$ in such a way that
\begin{equation}\label{iso} \int_{\C^{2n}} | f\ast_\lambda q_\lambda(z,w)|^2 w_\lambda(y,v) \, e^{\lambda (u \cdot y- v \cdot x)} \, dz dw  = c_n \int_{\R^{2n}} |f(x,u)|^2 dx du.
\end{equation} Under these assumptions, the reproducing kernel for  the space $ \widetilde{\mathcal{B}}_\lambda(\C^{2n}, w_\lambda) $ takes the form 
$$ K((z,w),(z^\prime,w^\prime)) = e^{-\lambda ( w \cdot z^\prime- z \cdot w^\prime)} q_\lambda \ast_\lambda q_\lambda(z-\bar{z^\prime},w-\bar{w^\prime}).$$
With these notations, we  formulate and prove  the following analogue of Theorem \ref{B-E-1}.\\

\begin{thm}\label{B-H-1}  Let  $ q \in L^1(\He^n) $ be such that for every  $ \lambda $ the image of $ L^2(\R^{2n}) $ under the transform $ \widetilde{T}_{q^\lambda}f(z,w) =  f \ast_\lambda q^\lambda(z,w) $ is a twisted Bergman  space $ \widetilde{\mathcal{B}}_\lambda(\C^{2n}, w_\lambda) $  where $ w_\lambda$ is radial. Then there is no nontrivial $ f \in L^2(\He^n) $ such that for all $ \lambda \in \R^\ast$
$$ \int_{\R^{2n}} | f^\lambda \ast_\lambda q^\lambda(y,v)| \, \sqrt{q^\lambda \ast_\lambda q^\lambda(2iy,2iv)} \, dy dv < \infty. $$ 
\end{thm} 

The twisted Bergman spaces were introduced by Kr\"otz-Thangavelu-Xu \cite{KTX} in the context of Segal-Bargmann transform $ T_bf(z,w,\zeta) = f \ast  p_b(z,w,\zeta) $ where $ p_b $ is the heat kernel associated to 
 the sublaplacian $ \mathcal{L} $ on $ \He^n.$  The associated twisted Bergman spaces correspond to the kernel $ p_b^\lambda $ and the weight function is given by 
 $ w_\lambda(y,v) = p_{2b}^\lambda(2y,2v).$ As in the Euclidean case the reproducing kernel is given by $ p_{2b}^\lambda(2iy,2iv).$ As a consequence, the condition in Theorem \ref{B-H-1} takes the form
 $$ \int_{\R^{2n}} | f^\lambda \ast_\lambda p_b^\lambda(y,v)| \, \sqrt{p_{2b}^\lambda(2iy,2iv)} \, dy dv < \infty. $$ 
Recall that $ \widehat{p_b}(\lambda) = e^{-b H(\lambda)} $ where $ H(\lambda) = -\Delta+\lambda^2 |x|^2 $ is the scaled Hermite operator on $ \R^n$  and Hardy's theorem for the Heisenberg group is formulated in terms of the Hermite semigroup $ e^{-bH(\lambda)}.$  The following analogue of Hardy's theorem proved in \cite{T2} follows from Theorem \ref{B-H-1} as an easy corollary. \\
 
 \begin{thm}[Hardy]\label{Hardy} 
Suppose a function $ f $ on $ \He^n $ satisfies the following two conditions: (i) $ |f(x,u,t)| \leq p_a(x,u,t) $ (ii)  $\hat{f}(\lambda)^\ast \hat{f}(\lambda) \leq C e^{-2b H(\lambda)}.$ Then $ f = 0 $ whenever $ a < b.$\\
\end{thm}

Ideally, we would like to replace  the condition (i) by $ |f^\lambda(x,u)| = c_\lambda \,p_a^\lambda(x,u) $ to arrive at the same conclusion.  We remark that this has been already proved by other means, see  Theorem 2.9.4 in \cite{T2}. However, we are not able to verify the Beurling condition in this situation. The conclusion remains valid even if we  replace (i)  by $ f(x,u,t) = g \ast_3 p_a(x,u,t) $ for some function $ g $ on $ \He^n $ for which $ g^\lambda \in L^\infty(\R^{2n}) $ for all $ \lambda.$ Here we have used the notation $ g \ast_3 h $ to stand for the convolution in the central variable. By taking the inverse Fourier transform in the central variable, we see that $ |f^\lambda(x,u)| = |g^\lambda(x,u)| p_a^\lambda(x,u) \leq \| g^\lambda\|_\infty \, p_a^\lambda(x,u)$ which is enough to prove the above theorem by other means. Moreover, by a theorem  of Douglas \cite{RD}, condition (ii) takes the form $ \widehat{f}(\lambda) = B_\lambda e^{-bH(\lambda)} $ for some bounded linear operator $ B_\lambda.$  By replacing $ g^\lambda $ by $ L^2 $ functions and $ B_\lambda $ by Hilbert-Schmidt operators, we get an analogue of Cowling-Price theorem on the Heisenberg group.  For the general $ L^p-L^q $ version  of Cowling-Price theorem, with a different proof,  we refer to the work  Parui-Thangavelu \cite{PT}.\\

 \begin{thm}[Cowling-Price]\label{C-P} 
Suppose a function $ f $ on $ \He^n $ satisfies the following two conditions:  for every $ \lambda \in \R^\ast,$ (i) $ f(x,u,t)  =  h(x,u,t) \, p_a(x,u,t) $ (ii)  $ \hat{f}(\lambda) =  B_\lambda \,e^{-b H(\lambda)}$ where $ h \in L^2(\He^n) $ and $ B_\lambda $ are Hilbert-Schmidt.  Then $ f = 0 $ whenever $ a < b.$\\
\end{thm}

\begin{rem} As every Hilbert-Schmidt operator  is the Weyl transform of an $ L^2(\R^{2n})$ function we can state the conditions of the above theorem as: (i)  $ f = h  p_a,$ and (ii)  $  f = g \ast p_b, $ where $ g,  h  \in L^2(\He^n).$ We also remark that the case $ a = b $ remains open.
\end{rem}

The same ideas lead to an interesting version of Beurling's theorem  for Hermite expansions.  Given $ f \in L^2(\R^n) $ we let $ P_kf $ stand for the orthogonal projection of $ f $ onto the eigenspace of the Hermite operator $ H $ with eigenvalue $ (2k+n).$ The Plancherel theorem for Hermite expansions reads as
\begin{equation} \| f\|_2^2 = \sum_{k=0}^\infty  \|P_kf\|_2^2, \,\,\, f(x) = \sum_{k=0}^\infty P_kf(x).
\end{equation}
Let $ L_k^{n-1}(t) $ stand for the Laguerre polynomials of type $ (n-1) $ and define the Laguerre functions $ \varphi_k^{n-1}(z) $ on $ \C^n$ by
$$ \varphi_k^{n-1}(z) = L_k^{n-1}(\frac{1}{2} z^2) e^{-\frac{1}{4} z^2},\,\,\, z^2 = \sum_{j=1}^n z_j^2.$$
Then we can prove the following theorem for Hermite expansions:

\begin{thm}[Beurling-Hermite]\label{beur-her} There is no non trivial $ f \in L^2(\R^n) $ for which
\begin{equation}\label{cond-beur-her}
\sum_{k=0}^\infty \int_{\R^n} |f(y)| \|P_kf\|_2 \sqrt{ \varphi_k^{n-1}(2iy)}  dy < \infty.
\end{equation}
\end{thm}

Thus we see that the role of $ e^{|(y, \xi)|} $ is played by $ \sqrt{ \varphi_k^{n-1}(2iy)} $ which is natural in view of Gutzmer's formula for the Hermite expansions. As an immediate corollary of this theorem we can obtain Hardy's theorem for Hermite expansions.\\

\begin{cor}[Hardy-Hermite] Suppose $ f \in L^2(\R^n) $ satisfies the following two conditions:
$$  |f(x)| \leq C e^{-\frac{1}{2} a |x|^2},\,\,\,\,  \|P_kf\|_2 \leq C e^{-\frac{1}{2} b(2k+n)}$$
for some $ a , b >0.$ Then $ f = 0 $ whenever $ a \tanh b >1.$
\end{cor}

\begin{rem} For a different proof of this result  see Theorem 1.4.7 in \cite{T2} where even the equality case has been considered. However, we cannot deduce the equality case from the Beurling's theorem we have proved.
\end{rem}

Here is the plan of the paper. After recalling the preliminaries on the Heisenberg group and collecting all the relevant results needed  in the next section, we will prove  the  theorems in Section 3.  

\section{Preliminaries on the Heisenberg group}

\subsection{Heisenberg group and the Fourier transform}
We develop the required background for the Heisenberg group.  General references for this section are the monographs \cite{T1},\cite{T2} and the book \cite{GF} of Folland.
Let $\mathbb{H}^n:=\mathbb{C}^n\times\mathbb{R}$ be the $(2n+1)$- dimensional Heisenberg group with the group law
	 
	$$(z, t)(w, s):=\big(z+w, t+s+\frac{1}{2}\Im(z.\bar{w})\big),\ \forall (z,t),(w,s)\in \mathbb{H}^n.$$ This is a step two nilpotent Lie group where the Lebesgue measure $dzdt$ on $\mathbb{C}^n\times\mathbb{R}$ serves as the Haar measure. In this article, we identify $ \C^n $ with  $ \R^{2n} $ and write the  elements of $ \He^n $ as $ (x,u,t).$  The group law takes the form
$$ (x,u,t) (y,v,s) = (x+y, u+v, t+s+\frac{1}{2}(u\cdot y-v \cdot x)).$$ The representation theory of $\mathbb{H}^n$ is well-studied in the literature. In order to define Fourier transform, we use the Schr\"odinger representations as described below.  \\
	
	For each non zero real number $ \lambda $ we have an infinite dimensional representation $ \pi_\lambda $ realised on the Hilbert space $ L^2( \R^n).$ These are explicitly given by
	$$ \pi_\lambda(x,u,t) \varphi(\xi) = e^{i\lambda t} e^{i\lambda(x \cdot \xi+ \frac{1}{2}x \cdot u)}\varphi(\xi+u),\,\,\,$$
	where  $ \varphi \in L^2(\R^n).$ These representations are known to be  unitary and irreducible. Moreover, by a theorem of Stone and Von-Neumann, (see e.g., \cite{GF})  upto unitary equivalence these account for all the infinite dimensional irreducible unitary representations of $ \mathbb{H}^n $ which act as $e^{i\lambda t}I$ on the center. Also there is another class of finite dimensional irreducible representations which   do not contribute to the Plancherel measure. Hence   we will not describe them here.\\
	
	The Fourier transform of a function $ f \in L^1(\mathbb{H}^n) $ is the operator valued function  defined on the set of all nonzero reals, $ \R^\ast $ given by
	$$ \widehat{f}(\lambda) = \int_{\mathbb{H}^n} f(x,u,t) \pi_\lambda(x,u,t)  dx du dt .$$  Note that $ \widehat{f}(\lambda) $ is a bounded linear operator on $ L^2(\R^n).$ It is known that when $ f \in L^1 \cap L^2(\mathbb{H}^n) $ its Fourier transform  is actually a Hilbert-Schmidt operator and one has
	\begin{equation}  \int_{\mathbb{H}^n} |f(x,u,t)|^2 dx du dt = (2\pi)^{-n-1}\int_{-\infty}^\infty \|\widehat{f}(\lambda)\|_{HS}^2 |\lambda|^n d\lambda  
	\end{equation}
	where $\|.\|_{HS}$ denotes the Hilbert-Schmidt norm. 
	The above allows us to extend  the Fourier transform as a unitary operator between $ L^2(\mathbb{H}^n) $ and the Hilbert space of Hilbert-Schmidt operator valued functions  on $ \R $ which are square integrable with respect to the Plancherel measure  $ d\mu(\lambda) = (2\pi)^{-n-1} |\lambda|^n d\lambda.$\\
	
	 For suitable functions $f$ on $\He^n$ we also have the following inversion formula: with the notation $ \pi_\lambda(x,u) = \pi_\lambda(x,u,0) $
	\begin{equation} f(x,u,t)=\int_{-\infty}^{\infty}   e^{-i \lambda t} \tr(\pi_{\lambda}(x,u)^*\widehat{f}(\lambda))d\mu(\lambda). \end{equation}\\
	From the  definition $ \pi_\lambda(x,u,t) = e^{i\lambda t} \pi_\lambda(x,u) $ and hence it follows that 
	$$\widehat{f}(\lambda)=\int_{\C^n}f^{\lambda}(x,u)\pi_{\lambda}(x,u) dx du $$ where 
	$f^{\lambda}$ stands for the inverse Fourier transform of $f$ in the central variable:
	$$f^{\lambda}(x,u):=\int_{-\infty}^{\infty}e^{i\lambda t}f(x,u,t)dt.$$
	This suggests that for functions  $g$ on $\C^n$, we consider the following  operator 
		$$ \pi_{\lambda}(g):=\int_{\C^n}g(z)\pi_{\lambda}(x,u) dx du$$ 
		known as the Weyl transform of $ g $. With these notations we note that  $\widehat{f}(\lambda)=\pi_{\lambda}(f^{\lambda}).$   We have the  Plancherel formula for the Weyl transform  (See \cite[2.2.9, page no-49]{T2}):  
		$$ \|g\|_2^2 =  (2\pi)^{-n} |\lambda|^n \|\pi_{\lambda}(g)\|^2_{HS}, $$ 
		for $ g\in L^2(\mathbb{\R}^{2n}) $  and the inversion formula reads as 
		\begin{equation}
			\label{weyl-inv}
			g(x,u) = (2\pi)^{-n}  |\lambda|^{n} \tr (\pi_\lambda(x,u)^\ast \pi_\lambda(g)).
		\end{equation}   
		We will make use of  these properties in the proof  of  Beurling's theorem on $ \He^n.$\\
		
The convolution on the Heisenberg group gives rise to a family of convolutions on $ \C^n$ called twisted convolutions. Recall that
$$ f \ast g(x,u,t) =\int_{\He^n} 	f\big(x-u,y-v,t-s-\frac{1}{2} (u \cdot y- v \cdot x)\big)	g\big(y,v,s\big)  du dv ds.$$
By calculating the inverse Fourier transform of $ f \ast g(x,u,t) $ in the last variable, we get
$$  \big(f \ast g\big)^\lambda(x,u) =  \int_{\R^{2n}} f^\lambda(x-u,y-v) g^\lambda(y,v) e^{i \frac{\lambda}{2}(u\cdot y- v \cdot x)}  dy dv = f^\lambda \ast_\lambda g^\lambda(x,u).$$
The integral on the right is called the $ \lambda$-twisted convolution of $ f^\lambda $ with $ g^\lambda.$ We call the function $f^\lambda(x-u,y-v)  e^{i \frac{\lambda}{2}(u\cdot y- v \cdot x)} $ the $ \lambda$-twisted translate of $ f^\lambda$ by $ (y,v).$ The well known relation $ \widehat{f \ast g}(\lambda) = \widehat{f}(\lambda) \widehat{g}(\lambda) $ leads to the formula $ \pi_\lambda( f^\lambda \ast_\lambda g^\lambda) = \pi_\lambda(f^\lambda) \pi_\lambda(g^\lambda).$

\subsection{Hermite and Laguerre  functions} We are tempted to define   $ \pi_\lambda(z,w) $  for complex values $ (z,w) \in \C^{2n} $  by the action
$$ \pi_\lambda(z,w) \varphi(\xi) =  e^{i\lambda(z \cdot \xi+ \frac{1}{2} z \cdot w)}\varphi(\xi+w).\,\,\,$$
This will make sense only if  $ \varphi(z) $ can be defined for complex $ z $ and the resulting function $e^{i\lambda(z \cdot \xi+ \frac{1}{2} z \cdot w)}\varphi(\xi+w)$ also belongs to $ L^2(\R^n).$ Fortunately, there is a dense subspace of $ L^2(\R^n) $ on which we can define $ \pi_\lambda(z,w) .$  This subspace consists of finite linear combinations of the Hermite functions $  \Phi_\alpha^\lambda(\xi) = |\lambda|^{n/4} \Phi_\alpha(|\lambda|^{1/2} \xi) $ where
$$ \Phi_\alpha(\xi) = c_\alpha \, H_\alpha(\xi) e^{-\frac{1}{2} |\xi|^2} .$$ Here $H_\alpha(\xi) $ are the Hermite polynomials on $ \R^n $ and $ c_\alpha $ are normalising constants. These functions $ \Phi_\alpha^\lambda, \alpha \in \mathbb{N}^n $ which are eigenfunctions of 
$ H(\lambda) = - \Delta+ \lambda^2 |\xi|^2$ with eigenvalues $ (2|\alpha|+n)|\lambda| $, form an orthonormal basis for $ L^2(\R^n).$ As $ H_\alpha $ are polynomials, we can extend $ \Phi_\alpha(\xi) $ to $ \C^n$ by setting $ \Phi_\alpha(z) = c_\alpha H_\alpha(z) e^{-\frac{1}{2}z^2} $ where $ z^2 = \sum_{j=1}^n z_j^2.$ With this we have
$$ \pi_\lambda(z,w) \Phi_\alpha^\lambda(\xi)  =  e^{i\lambda(z \cdot \xi+ \frac{1}{2} z \cdot w)}\Phi_\alpha^\lambda(\xi+w).\,\,\,$$
These functions $ \pi_\lambda(z,w) \Phi_\alpha^\lambda $ turn out to be in $ L^2(\R^n) $ and hence $ \pi_\lambda(z,w) $ are densely defined unbounded operators. The formal adjoint of $ \pi_\lambda(z,w) $ can be calculated and we have
$ \pi_\lambda(z,w)^\ast = \pi_\lambda(-\bar{z}, -\bar{w}).$ Thus for $ (y,v) \in \R^{2n}$ the operators $ \pi_\lambda(i(y,v)) $ are self-adjoint. \\

The $ L^2(\R^n) $ norm of $ \pi_\lambda(z,w) \Phi_\alpha^\lambda $ can be calculated in terms of the special Hermite functions 
$$ \Phi_{\alpha,\beta}^\lambda(x,u) = (2\pi)^{-n/2}  \big(\pi_\lambda(x,u) \Phi_\alpha^\lambda, \Phi_\beta^\lambda \big).$$
The special Hermite functions $\Phi_{\alpha,\beta}^\lambda(x,u) $ are expressible in terms of Laguerre functions. As a result, these functions can also defined for $ (z,w) \in \C^{2n} $ as entire functions. For our purposes, we just recall one important  formula. Let $ L_k^{n-1}(r),\, r > 0 $ stand for Laguerre polynomials of type $ (n-1).$ We let 
$$  \varphi_{k,\lambda}^{n-1}(x,u) = L_k^{n-1}\big(\frac{1}{2}|\lambda| (|x|^2+|u|^2)\big) e^{-\frac{1}{4} |\lambda| (|x|^2+|u|^2)} $$
which has a natural holomorphic extension to $ \C^n \times \C^n.$ We will make use of the formula
$$  (2\pi)^{-n/2}\,\varphi_{k,\lambda}^{n-1}(z,w) = \sum_{|\alpha|=k } \Phi_{\alpha,\alpha}^\lambda(z,w) $$ 
which expresses the Laguerre function $  \varphi_{k,\lambda}^{n-1}(z,w) $ in terms of  special Hermite functions.\\

The special Hermite functions $ \Phi_{\alpha,\beta}^\lambda, \alpha, \beta \in \mathbb{N}^n $ form an orthonormal basis for $ L^2(\R^{2n}).$ The special Hermite expansion of any function $ g \in L^2(\R^{2n}) $ can be put in the compact form
$$ g(x,u) = (2\pi)^{-n} \, |\lambda|^n \, \sum_{k=0}^\infty g \ast_\lambda  \varphi_{k,\lambda}^{n-1}(x,u) .$$
In terms of this expansion, the inversion formula for the Fourier transform on $ \He^n $ can be written in the form
$$ f(x,u,t) = (2\pi)^{-n-1} \int_{-\infty}^\infty   e^{-i \lambda t} \Big( \sum_{k=0}^\infty f^\lambda \ast_\lambda  \varphi_{k,\lambda}^{n-1}(x,u) \Big) |\lambda|^n d\lambda.$$

\subsection{ Gutzmer's formula}  As mentioned in the introduction, we let  $ \sigma(x,u) $ stand for the action of $ U(n) $ identified with $ K = Sp(n,\R) \cap O(2n,\R) $ on $ \R^{2n}.$  The symplectic form $ [(x,u), (y,v)] = (u \cdot y - v \cdot x) $ on $ \R^{2n} $ is invariant under the action of $ K.$  The action of this group  has a natural extension to $ \C^n \times \C^n $ and so it make sense to talk about $ \sigma (z, w) $ for $ z, w \in \C^n.$  In Theorem \ref{B-H} we have imposed the condition
\begin{equation}\label{condi-1-1}
  \int_{U(n)}  \|\pi_\lambda(\sigma(z,w))^\ast \widehat{f}(\lambda)\|_{HS}^2  \, d\sigma  < \infty
  \end{equation}
on  the Fourier transform  $ \hat{f}(\lambda) .$ From this we would like to conclude that $ f^\lambda(x,u) $ can be extended to $ \C^{2n} $ as a holomorphic function. The integral appearing above can be explicitly computed  using Gutzmer's formula for special Hermite expansions proved in \cite{T3}, see Theorem 6.2 (Caution: the integral over $ U(n) $ is missing there!) 

\begin{thm}[Gutzmer]\label{Gutz}
Let $ g \in L^2(\R^{2n}) $ has a holomorphic extension $ G $ to $ \C^{2n}.$ Then  we have
$$ \int_{U(n)}   \int_{\R^{2n}}  |G((x,u)+i \sigma(y,v))|^2 e^{\lambda(u \cdot y- v \cdot x)}  \, dx du d\sigma $$
$$ = c_n  \sum_{k=0}^\infty \frac{k!(n-1)!}{(k+n-1)!} \varphi_{k,\lambda}^{n-1}(2iy,2iv) \|  g \ast_\lambda \varphi_{k,\lambda}^{n-1} \|_2^2.$$
\end{thm}

Using the above result, we can calculate the integral (\ref{condi-1-1}).

\begin{prop}\label{conseq-1} For any $ f \in L^2(\He^n) $ we have the identity
$$  \int_{U(n)}  \|\pi_\lambda(\sigma(z,w)) \hat{f}(\lambda)\|_{HS}^2  \, d\sigma  $$
$$=  e^{-\lambda(u\cdot y-x \cdot v)} \sum_{k=0}^\infty \frac{k!(n-1)!}{(k+n-1)!} \varphi_{k,\lambda}^{n-1}(2iy,2iv) \| f^\lambda \ast_\lambda \varphi_{k,\lambda}^{n-1}\|_2^2$$
with the understanding that either the left hand side or the right hand side is finite.
\end{prop}

\begin{proof}In order to calculate (\ref{condi-1-1}) let us write it as the sum
$$   \sum_{k=0}^\infty  \int_{U(n)} \|\pi_\lambda(\sigma(z,w))^\ast \widehat{f}(\lambda)P_k(\lambda)\|_{HS}^2    \, d\sigma.$$
As $ \widehat{f}(\lambda) = \pi_\lambda(f^\lambda) $ and $ \pi_\lambda(\varphi_{k,\lambda}^{n-1}) = (2\pi)^n |\lambda|^{-n} P_k(\lambda) $  (see \cite{T3}, Proposition 2.3.3) we have 
$$ \widehat{f}(\lambda)P_k(\lambda) = (2\pi)^{-n} |\lambda|^n  \pi_\lambda(f^\lambda \ast_\lambda \varphi_{k,\lambda}^{n-1}).$$
We note that the function $ g(x,u) = f^\lambda \ast_\lambda \varphi_{k,\lambda}^{n-1}(x,u) $ extends to $ \C^{2n} $ as a homomorphic function and hence the relation
$$ \pi_\lambda(x,u)^\ast \pi_\lambda(g) = \pi_\lambda(\tau_{(x,u)}^\lambda g), \,\,\, \tau_{(x,u)}^\lambda g(x^\prime,u^\prime) = g(x^\prime-x,u^\prime-u)e^{-i\frac{\lambda}{2}(u\cdot x^\prime- u^\prime\cdot x)} $$
continues to hold even when $ (x,u) $ is replaced by $ (z,w).$ Therefore, by the Plancherel theorem for the Weyl transform, we have
$$ (2\pi)^{-n} |\lambda|^{n} \int_{\R^{2n}} |\tau_{(z,w)}^\lambda g(x^\prime,u^\prime)|^2 \, dx^\prime du^\prime =  \|\pi_\lambda((z,w))^\ast \widehat{f}(\lambda)P_k(\lambda)\|_{HS}^2.     $$ 
Recalling the definition of $ \tau_{(z,w)}^\lambda g$ and making some change of variables, we see that
$$   \|\pi_\lambda((z,w))^\ast \widehat{f}(\lambda)P_k(\lambda)\|_{HS}^2 $$  
$$ = (2\pi)^{-n} |\lambda|^{n}   e^{-\lambda(u \cdot y- v \cdot x)} \int_{\R^{2n}} |g(x^\prime-iy,u^\prime-iv)|^2 e^{- \lambda(u^\prime\cdot y -x^\prime \cdot v)}   dx^\prime du^\prime. $$ 
As the Lebesgue measure on $ \R^{2n} $ as well the symplectic form $ [(x,u),(y,v)] = (u\cdot y-v \cdot x) $ is   invariant under the action of $ U(n) $ the  integral in the above equation becomes
$$  \int_{\R^{2n}} |g\big((x^\prime,u^\prime)+i\sigma(-y,-v)\big)|^2 e^{- \lambda(u^\prime\cdot y -x^\prime \cdot v)}   dx^\prime du^\prime. $$ 
We can now appeal to Theorem \ref{Gutz} and use the relation 
$$ \varphi_{k,\lambda}^{n-1} \ast_\lambda \varphi_{j,\lambda}^{n-1} =  \delta_{jk} (2\pi)^n |\lambda|^{-n} \varphi_{k,\lambda}^{n-1} $$  to arrive at the identity
$$   \|\pi_\lambda((z,w))^\ast \widehat{f}(\lambda)P_k(\lambda)\|_{HS}^2  $$
$$= e^{-\lambda(u\cdot y-x \cdot v)}  \frac{k!(n-1)!}{(k+n-1)!} \varphi_{k,\lambda}^{n-1}(2iy,2iv) \| f^\lambda \ast_\lambda \varphi_{k,\lambda}^{n-1}\|_2^2.$$  
This completes the proof of the proposition.
\end{proof}

\subsection{ The sublaplacian and the associated heat kernel}  We let $ \mathfrak{h}_n $ stand for the Heisenberg Lie algebra consisting of left invariant vector fields on $ \mathbb{H}^n .$  A  basis for $ \mathfrak{h}_n $ is provided by the $ 2n+1 $ vector fields
	 $$ X_j = \frac{\partial}{\partial{x_j}}+\frac{1}{2} u_j \frac{\partial}{\partial t}, \,\,U_j = \frac{\partial}{\partial{u_j}}-\frac{1}{2} x_j \frac{\partial}{\partial t}, \,\, j = 1,2,..., n $$
	 and $ T = \frac{\partial}{\partial t}.$  These correspond to certain one parameter subgroups of $ \mathbb{H}^n.$ The sublaplacian on $\He^n$ is defined by $\mathcal{L}:=-\sum_{j=1}^{\infty}(X_j^2+U_j^2) $ which is given explicitly by
	 $$\mathcal{L}=-\Delta_{\R^{2n}}-\frac{1}{4}|(x,u)|^2\frac{\partial^2}{\partial t^2}+N\frac{\partial}{\partial t}$$ where  $\Delta_{\R^{2n}}$ stands for the Laplacian on $\R^{2n}$ and $N$ is the rotation operator defined by 
	 $$N=\sum_{j=1}^{n}\left(x_j\frac{\partial}{\partial u_j}-u_j\frac{\partial}{\partial x_j}\right).$$ This is a sub-elliptic operator and homogeneous of degree $2$ with respect to the non-isotropic dilation given by $\delta_r(x,u,t)=(rx,ru,r^2t).$\\
	 
	 Along with the sublaplacian, we also  consider the special Hermite operator $L_{\lambda}$ defined by the relation $(\mathcal{L}f)^{\lambda}(x,u)=L_{\lambda}f^{\lambda}(x,u).$ It follows that $ L_\lambda $  is explicitly given by 
	 $$L_\lambda = -\Delta_{\R^{2n}}+\frac{1}{4} \lambda^2 |(x,u)|^2+ i \lambda N.$$ It turns out that  $f^{\lambda}\ast_{\lambda}\varphi_{k,\lambda}(x,u)$ are eigenfunctions of $L_{\lambda}$ with eigenvalues $(2k+n)|\lambda|$ and hence  we have the following expansion 
	 $$L_{\lambda}f^{\lambda}(x,u)=(2\pi)^{-n} |\lambda|^n \sum_{k=0}^\infty ((2k+n)|\lambda|)\,   f^{\lambda}\ast_{\lambda}\varphi_{k,\lambda}^{n-1}(x,u)$$ 
	 leading to the following spectral decomposition of $\mathcal{L}$:
	 \begin{equation}
	 	\label{lspec}
	 	\mathcal{L}f(x,u,t)= (2\pi)^{-n-1}  \int_{-\infty}^\infty  e^{-i\lambda t} \Big( \sum_{k=0}^\infty ((2k+n)|\lambda|)\,   f^{\lambda}\ast_{\lambda}\varphi_{k,\lambda}^{n-1}(x,u)\Big) |\lambda|^n d\lambda. 
	 \end{equation}	
	 The heat kernel $ p_a(x,u,t), a >0$ associated to the sublaplacian is defined by
	 $$ p_a(x,u,t) = (2\pi)^{-n-1}  \int_{-\infty}^\infty  e^{-i\lambda t} \Big( \sum_{k=0}^\infty e^{-a (2k+n)|\lambda|} \varphi_{k,\lambda}^{n-1}(x,u)\Big) |\lambda|^n d\lambda.$$
	 Though the heat kernel $ p_a $ is not known explicitly, $ p_a^\lambda(x,u) $ can be calculated. In fact,
 $$ p_a^\lambda(x,u) = (2\pi)^{-n}  \sum_{k=0}^\infty e^{-a (2k+n)|\lambda|} \varphi_{k,\lambda}^{n-1}(x,u)$$
 and the above series can be summed using a generating function identity for Laguerre functions to get
 \begin{equation}  p_a^\lambda(x,u) = (4\pi)^{-n}  \Big( \frac{\lambda}{\sinh a\lambda}\Big)^n  e^{-\frac{1}{4}\lambda (\coth a\lambda)(|x|^2+|u|^2)}.
 \end{equation}
  In terms of this, the semigroup generated by $ L_\lambda $ is given by $ e^{-aL_\lambda} g(x,u) = g \ast_\lambda p_a^\lambda(x,u).$\\
 
 From the explicit formula for $ p_a^\lambda $ it is easy to see that $ g \ast_\lambda p_a^\lambda(x,u) $ has a holomorphic extension to $ \C^{2n}.$ The problem of characterising the image of $ L^2(\R^{2n}) $ under the transform $ T_a^\lambda g(z,w) = g \ast_\lambda p_a^\lambda(z,w) $ has been studied in Kr\"otz et al \cite{KTX} where the authors have proved the following result.

\begin{thm}[Kr\"otz et al]\label{twist-berg}An entire function $ G(z,w) $ on $ \C^{2n} $  is square integrable  with respect to the weight function $ e^{\lambda( u\cdot y-v \cdot x)} p_{2a}^\lambda(2y,2v) $
if and only if $ G(x,u) = g \ast_\lambda p_a^\lambda(x,u) $ for some $ g \in L^2(\R^{2n}).$ Moreover, for any $ g \in L^2(\R^{2n}) $ we have
$$ \int_{\C^{2n}} | g \ast_\lambda p_a^\lambda(z,w)|^2 e^{\lambda( u\cdot y-v \cdot x)} p_{2a}^\lambda(2y,2v) dz dw = c_\lambda \int_{\R^{2n}} |g(x,u)|^2 dx du.$$
\end{thm}

Thus we are led to consider the twisted Bergman space $ \widetilde{\mathcal{B}}_\lambda(\C^{2n}, w_\lambda) $ consisting of all entire functions $ G(z,w) $ on $ \C^{2n} $  which are  square integrable  with respect to the weight function $ p_{2a}^\lambda(2y,2v)e^{\lambda( u\cdot y-v \cdot x)}.$ In the above setting   the map $ T_a^\lambda $ is an isometry up to a constant multiple from $ L^2(\R^{2n}) $ onto $ \widetilde{\mathcal{B}}_\lambda(\C^{2n}, w_\lambda) .$\\

We now consider the spaces $ \widetilde{\mathcal{B}}_\lambda(\C^{2n}, w_\lambda) $ associated to the weight function $ w_\lambda(y,v).$ From the form of the weight, it is easy to see that these spaces are invariant under twisted translations. These are reproducing kernel Hilbert spaces and hence possess a reproducing kernel which we denote by $ K((z,w), (z^\prime,w^\prime)).$ From general theory it follows that every $ G \in  \widetilde{\mathcal{B}}_\lambda(\C^{2n}, w_\lambda) $ satisfies the point-wise estimate $ |G(z,w)|^2 \leq C K((z,w),(z,w)).$ In the case of the space associated to the transform $ T_a^\lambda $ the reproducing kernel is given by
$$  e^{-i \frac{\lambda}{2}( w \cdot z^\prime- z \cdot w^\prime)}p_a^\lambda \ast_\lambda p_a^\lambda(z-\bar{z^\prime}, w-\bar{w^\prime}) = e^{-i \frac{\lambda}{2}( w \cdot z^\prime- z \cdot w^\prime)} p_{2a}^\lambda(z-\bar{z^\prime}, w-\bar{w^\prime}).$$ Thus for any $ g \in L^2(\R^{2n}) $ we get the estimate
$$ | g \ast_\lambda p_a^\lambda(z,w)|^2 \leq C \, p_{2a}^\lambda(2iy,2iv) .$$
If we assume that $ \widetilde{\mathcal{B}}_\lambda(\C^{2n}, w_\lambda) $ occurs as the image of a transform $ T_{q_\lambda} g = g \ast_\lambda q_\lambda ,$ then we can calculate the reproducing kernel in terms of $ q_\lambda$ and we can prove the estimate
$$ | g \ast_\lambda q_\lambda(z,w)|^2 \leq C \,  q_\lambda \ast_\lambda q_\lambda(2iy,2iv) e^{\lambda(u\cdot y-v \cdot x)}.$$ This is precisely what motivated us to formulate Theorem \ref{B-H-1}.\\

\subsection{Constructing examples of twisted Bergman spaces}  
In the introduction we have mentioned how one can construct a weighed Bergman space on $ \C^n $ which is invariant under the action of $ \R^n$ starting from a weight function $ w .$ In this subsection, we describe such a procedure to construct twisted Bergman spaces on $ \C^{2n}.$  For that purpose, as in the Euclidean case, we make use of the Gutzmer's formula described in the previous subsection.\\

We start with a positive weight function $ w_\lambda(y,v) $ on $ \R^{2n} $ which we assume to be radial. Consider the sequence $ C_\lambda(k) $ defined by
$$ C_\lambda(k)  =  \frac{k!(n-1)!}{(k+n-1)!} \, \int_{\R^{2n}}  w_\lambda (y,v)  \varphi_{k,\lambda}^{n-1}(2iy,2iv) dy dv .$$
Under the assumption that $ C_\lambda(k) $ has enough decay, we can define  the radial function
$$ q_\lambda(x,u) = \sum_{k=0}^\infty C_\lambda(k)^{-1/2} \, \varphi_{k,\lambda}^{n-1}(x,u) .$$ 
Recall that these functions $  \varphi_{k,\lambda}^{n-1}(x,u) $ extend to $ \C^{2n} $ as entire functions and satisfy the estimates
$$ | \varphi_{k,\lambda}^{n-1}(z,w)|^2 \leq C_t \, e^{2t(2k+n)|\lambda|} \, p_{2t}^\lambda(2iy,2iv) $$ for any $ t >0,$ (see \cite{KTX}.) If we know that $ C_\lambda(k)^{-1}  \leq c_\lambda\, e^{-t (2k+n)^s |\lambda|^s} $ for some $ t > 0,  s >1,$  then the function $ q_\lambda$  defines a holomorphic function on the whole of $ \C^{2n}.$  It then follows that  for any $ f \in L^2(\R^{2n}) $ the function $ f \ast_\lambda q_\lambda $ has a holomorphic extension to $ \C^n \times \C^n $  and by Gutzmer's formula we have
$$ \int_{\C^{2n}} |f \ast_\lambda q_\lambda(z,w)|^2 w_\lambda(y,v) \, e^{\lambda (u \cdot y- v \cdot x)} \, dz dw $$
$$  = c_{n,\lambda} \,  \sum_{k=0}^\infty  \| f \ast_\lambda  \varphi_{k,\lambda}^{n-1} \|_2^2 \, \, C_\lambda(k)^{-1}\,  \,  \frac{k!(n-1)!}{(k+n-1)!} \, \int_{\R^{2n}} w_\lambda(y,v)   \varphi_{k,\lambda}^{n-1}(2iy,2iv)  dy dv.$$
By our choice of $ C_\lambda(k) $ the right hand side of the above reduces to  $ \| f\|_2^2 $ proving that the image of $ L^2(\R^{2n}) $ under the transform $ T_{q_\lambda}: f \rightarrow f \ast_\lambda q_\lambda(z,w) $ is a twisted Bergman space with weight function $ w_\lambda(y,v).$ The reproducing kernel for this 
space is given by 
$$ K_\lambda((z,w), (z^\prime,w^\prime)) =  e^{-i \frac{\lambda}{2}( w \cdot z^\prime- z \cdot w^\prime)} q_\lambda \ast_\lambda q_\lambda(z-\bar{z^\prime}, w- \bar{w^\prime}) .$$

When we take the  weight function $ w_\lambda(y,v) = p_{2t}^\lambda(2y,2v)$ it is known that $ C_\lambda(k) = e^{2t(2k+n)|\lambda|} $ and hence $ q_\lambda(x,u) = p_t^\lambda(x,u) $  which is coming from the Segal-Bargmann transform $ T_t f = f \ast p_t $ on $ \He^n.$ By taking superpositions of $ p_t^\lambda$ we can construct other examples of $ w_\lambda $ and the kernels $ q_\lambda.$  For example, for any $ 1 < s < \infty,$ let us define

$$ w_\lambda(y,v) = \int_0^\infty t^n e^{-\frac{1}{s} t^s} p_t^\lambda(2y,2v) dt . $$
We recall the following result proved in \cite{T3}, (see Lemma 6.3):
$$ \frac{k! (n-1)!}{(k+n-1)!} \int_{\R^{2n}} p_t^\lambda(iy,iv) \varphi_{k,\lambda}^{n-1}(iy,iv) dy dv = c_n \, e^{t(2k+n)|\lambda|}.$$
Integrating $ w_\lambda $ against $  \varphi_{k,\lambda}^{n-1}(2iy,2iv) $ and making use of the above formula  we see that
$$ \frac{k! (n-1)!}{(k+n-1)!} \, \int_{\R^{2n}}  w_\lambda (y,v)   \varphi_{k,\lambda}^{n-1}(2iy,2iv) dy dv = c_n \, \int_0^\infty t^n e^{-\frac{1}{s} t^s + (2k+n)|\lambda| t} dt\,\, .$$
Let $  A = (2k+n)|\lambda|$ and define $ s^\prime $ as the index conjugate to $ s.$ By  writing $ (\frac{1}{s}t^s-At) = \big(\frac{1}{s} t^s + \frac{1}{s^\prime} A^{s^\prime}- At \big) -\frac{1}{s^\prime} A^{s^\prime}$ and making the change of variables $ t \rightarrow A^{s^\prime /s} t$  the above integral becomes
$$ A^{(n+1) s^\prime /s} e^{\frac{1}{s^\prime}A^{s^\prime}} \int_0^\infty  t^n e^{-A^{s^\prime} ( \frac{1}{s} t^s+ \frac{1}{s^\prime} -t)}  dt. $$
As the function $ \psi(t) = (\frac{1}{s} t^s+ \frac{1}{s^\prime} -t) $ decreases from $ \frac{1}{s^\prime} $ to $ 0 $ on the interval $ [0,1] $ we can choose $ 0 < t_0 < 1 $ such that $ \psi(t_0) = \frac{1}{2s^\prime} $ and get the lower bound
$$ \int_0^\infty  t^n e^{-A^{s^\prime} ( \frac{1}{s} t^s+ \frac{1}{s^\prime} -t)}  dt \geq  e^{- \frac{1}{2s^\prime}A^{s^\prime}} \, \int_{t_0}^1 t^n dt = c \,e^{- \frac{1}{2s^\prime}A^{s^\prime}}. $$
This gives us the following  lower bound for the sequence $ C_\lambda(k): $ 
$$ C_\lambda(k)  \geq c_\lambda   \, ((2k+n)|\lambda|)^{(n+1)(s^\prime-s)} e^{ \frac{1}{2s^\prime} ((2k+n)|\lambda|)^{s^\prime}} .$$
Thus we see that $ C_\lambda(k)^{-1/2} $ has the required decay so that the function $ q_\lambda $ gives rise to a twisted Bergman space.\\

\section{Beurling's theorem and its consequences}  

\subsection{Hedenmalm's proof} For the convenience of the reader we recall Hedenmalms's proof from his paper \cite{HH}. 
The idea is very simple: under the assumption that $ f $ and $ \widehat{f} $ satisfy (\ref{B-C}), the function
$$ F(\zeta) = \int_{\R^n}\int_{\R^n} \bar{f}(y) \hat{f}(\xi) e^{i\zeta(y, \xi)} dy d\xi $$ defines a holomorphic function on the strip
$ S =  \{ \zeta: |\Im \zeta | < 1\} $ and extends continuously up to its closure. We also note that
 $$ F(\zeta) = (2\pi)^{n/2} \int_{\R^n}  \bar{f}(y) \, f(\zeta y) \,dy.$$When $ \zeta = r $ is real, owing to the inversion formula for the Fourier transform, we have
$$ F(r) = (2\pi)^{n/2} \int_{\R^n}  \bar{f}(y) \, f(ry) \,dy = r^{-n} \overline{F(r^{-1})}.$$
 By defining $ G(\zeta) = (1+\zeta^2)^{n/2} F(\zeta) $ we see that  functional relation $ F(r) = r^{-n}\overline{F(r^{-1})}$ translates into $ G(r) = \overline{G(r^{-1})}.$  The function $ G $ is holomorphic in a neighbourhood of  $ |\zeta| \leq 1$  save for $i $ and $ -i.$ By defining $ G_1(\bar{\zeta}) =  \overline{G(\frac{1}{\zeta})} $ we get another holomorphic function on $ |\zeta| > 1.$    These two  domains have an overlap along the real line where they coincide. Thus we see  that $ G $  extends holomorphically to the whole of $ \C $ save for $i $ and $ -i.$ However, these are removable singularities as $ F $ remains bounded in neighbourhoods of these points. Therefore, $ G $ is entire and bounded and hence reduces to a constant. But as $ F $ vanishes at infinity we conclude that $ F(\zeta) = 0 $ for all $ \zeta.$  As $ F(1) = \|f\|_2^2 = 0 ,$ this proves the theorem.\\
 
If we already know that $ f $ has a holomorphic extension to $ \C^n $ then the Buerling's condition can be replaced by the modified one (\ref{mod-b}). The above proof remains valid without any problem.

\subsection{Beurling's theorem on $ \He^n$}  We are now ready for proving Theorem \ref{B-H}. As in the Euclidean case we start with defining the function
$$  F_\lambda(\zeta) =   \int_{\R^{2n}}  \overline{f^\lambda(y,v)} f^\lambda(\zeta y, \zeta v)    dy dv  $$ for complex values of $ \zeta.$  For this to make sense, we have to first check if $ f^\lambda(x,u) $ can be holomorphically extended to some domain in $ \C^{2n} $ and the integral converges for all $ \zeta $ in some other domain in $ \C.$  Our hypothesis (\ref{condi-one}) when combined with Proposition \ref{conseq-1} leads to
\begin{equation}\label{conseq-2} \sum_{k=0}^\infty \frac{k!(n-1)!}{(k+n-1)!} \, \varphi_{k,\lambda}^{n-1}(2iy,2iv) \,  \| f^\lambda \ast_\lambda \varphi_{k,\lambda}^{n-1}\|_2^2 < \infty .
\end{equation}
We recall the following asymptotic behaviour of the Laguerre polynomials: Perron's formula for Laguerre polynomials
in the complex domain (see Theorem 8.22.3 in Szego \cite{GS}) states
\begin{equation}\label{perron}
L_k^\alpha(s) =\frac{1}{2} \pi^{-1/2} e^{s/2} (-s)^{-\alpha/2-1/4} k^{\alpha/2-1/4} e^{2 \sqrt{-sk}} \big( 1+O(k^{-1/2})\big)
\end{equation}
valid for $s $ in the complex plane cut along the positive real axis. From this we get
$$  C_1(\rho) (2k+n)^{\gamma_1} e^{c_1 \rho \sqrt{2k+n}} \leq L_k^{n-1}(-2\rho^2) e^{\rho^2} \leq C_2(\rho) (2k+n)^{\gamma_2} e^{c_2 \rho \sqrt{2k+n}} $$ as $ k $ tends to infinity.  In view of this, the finiteness of (\ref{conseq-2}) for all $ (y,v) $ implies that  for every $ \rho >0,$
$$ \| f^\lambda \ast_\lambda \varphi_{k,\lambda}^{n-1} \|_2 \leq C e^{-\rho \sqrt{(2k+n)|\lambda|}} .$$
This allows us to write  $ f^\lambda = g_\lambda \ast_\lambda P_\rho^\lambda $ where $ g_\lambda \in L^2(\R^{2n})$ and 
$$ P_\rho^\lambda(x,u) = (2\pi)^{-n}  \sum_{k=0}^\infty e^{-\rho\, \sqrt{(2k+n)|\lambda|}} \, \varphi_{k,\lambda}^{n-1}(x,u) $$
is the Poisson kernel associated to the operator $ L_\lambda.$ Observe that $ P_\rho^\lambda $ extends holomorphically to a tube domain $ \Omega_{\gamma \rho} = \{ (z,w): |(y,v)| < \gamma \rho \} $ for some $ \gamma >0.$ This is a 
consequence of the estimate
$$ |\varphi_{k,\lambda}^{n-1}(z,w) |^2 \leq C \frac{(k+n-1)!}{k! (n-1)!} e^{\lambda (u \cdot y - v \cdot x)} \varphi_{k,\lambda}^{n-1}(2iy,2iv) $$
(see Proposition 3.1 in \cite{T3}). Therefore, from  the factorisation $ f^\lambda = g_\lambda \ast_\lambda P_\rho^\lambda $ we infer that $ f^\lambda $ extends holomorphically to a tube domain $ \Omega_{\gamma \rho}.$ As this is true for any $ \rho$ we conclude that $ f^\lambda $ has a holomorphic extension to the whole of $ \C^{2n}.$\\

Thus the function $ \zeta \rightarrow f^\lambda(\zeta y, \zeta v) $ is holomorphic, but still the integral defining $ F_\lambda(\zeta) $ need not converge. From the inversion formula for the Weyl transform, we have
$$ f^\lambda(y,v) = (2\pi)^{-n} |\lambda|^n  \tr \big( \pi_\lambda(y,v)^\ast \widehat{f}(\lambda)\big),$$
from which we obtain, with $ \zeta = r+is$, the following expression:
$$f^\lambda(\zeta y, \zeta v) = (2\pi)^{-n} |\lambda|^n  \tr \big( \pi_\lambda(r(y,v))^\ast  \pi_\lambda(is(y,v)) \widehat{f}(\lambda)\big).$$
As $\pi_\lambda(r(y,v))^\ast $ are unitary operators, the above gives the estimate 
$$   |  f^\lambda(\zeta y, \zeta v) | =(2\pi)^{-n} |\lambda|^n  | \tr( \pi_\lambda( r(y,v))^\ast \, \pi_\lambda(is (y,v))\hat{f}(\lambda))|  \leq c_\lambda \|  \pi_\lambda(is (y,v)) \, \hat{f}(\lambda) \|_1. $$ 
By considering the radial majorant of the above, for $ |s| \leq 1$, we get the estimate
$$ |f^\lambda(\zeta y, \zeta v)| \leq c_\lambda \,  \sup_{|(y^\prime,v^\prime)| \leq |(y,v)| } \| \pi_\lambda((i(y^\prime,v^\prime))  \widehat{f}(\lambda) \|_1. $$
Therefore, in view of the hypothesis (\ref{mod-b-h}), it follows that $ F_\lambda(\zeta) $ defines a holomorphic function on the strip $ S =  \{ \zeta: |\Im \zeta | < 1\} $ which is continuous up to the boundary. We also note that for $ r > 0 $  the relation $ F_\lambda(r) = r^{-2n} \overline{F_\lambda(r^{-1})}.$ From this point we just need to repeat Hedenmalm's proof to conclude that $ f^\lambda = 0.$ As this is true true for all  $ \lambda \in \R^\ast,$  Theorem \ref{B-H} is proved. It is also clear that the hypothesis (\ref{mod-b-h}) can be replaced by (\ref{mod-b-h-1}).

\subsection{Hardy and Cowling-Price theorems}  We can now easily deduce Hardy and Cowling-Price theorems on the Heisenberg group from Theorem \ref{B-H}. In doing so, we will make use of the following theorem of  Douglas \cite{RD}. (Here we have stated only  part of the theorem which is relevant for us).

\begin{thm}[Douglas] Let $ A $ and $ B $ be two bounded linear operators on a Hilbert space $ \mathcal{H}.$ Then $ A A^\ast  \leq c^2 B B^\ast $ if and only if there is a bounded linear operator $ C $ such that $ A = BC.$
\end{thm}

In view of this result, the second assumption $ \widehat{f}(\lambda)^\ast \widehat{f}(\lambda)  \leq C e^{-2b H(\lambda)} $ of Hardy's theorem gives us $ \widehat{f}(\lambda) = B_\lambda e^{-bH(\lambda)} $ where $ B_\lambda $ are bounded.  As $ a < b $ we can choose $ b^\prime $ such that $ a < b^\prime < b$ and write $ \widehat{f}(\lambda) = B_\lambda^\prime e^{-b^\prime H(\lambda)} $ where now $ B_\lambda^\prime $ is Hilbert-Schmidt. This allows us to write $ f^\lambda = g_\lambda \ast_\lambda p_{b^\prime}^\lambda $ with $ g_\lambda \in L^2(\R^{2n}).$ Thus, $ f^\lambda $ belongs to the twisted Bergman space with weight function $ p_{2b^\prime}^\lambda(2y,2v).$  This leads us to the estimate
$$ |f^\lambda(z,w)|^2 \leq c_\lambda e^{\lambda(u \cdot y-v \cdot x)} p_{2b^\prime}^\lambda(2iy,2iv).$$ Thus we can take $ F_\lambda(y,v)^2 = p_{2b^\prime}^\lambda(2iy,2iv) $  and see if this verifies the condition (\ref{mod-b-h-1}).  The assumption on $ f $ namely $ |f(x,u,t)| \leq C p_a(x,u,t) $ gives the estimate
$$ |f^\lambda(y,v)| \leq C \int_{-\infty}^\infty p_a(y,v,t) dt \leq C_a  e^{-\frac{1}{4a} (|y|^2+|v|^2)} $$
which is not good enough for verifying (\ref{mod-b-h-1}) for all values of $ \lambda$ but only for small enough values. However, it serves the purpose when combined with an analyticity argument.\\

The following estimate for the heat kernel is known (see e.g Proposition 2.8.2 in \cite{T2}):
$$ |p_a(y,v,t)| \leq c_n  a^{-n-1} e^{-\frac{d}{a}(|(y,v)|^2+ |t|)} $$
for some $ d>0.$ As a consequence, the estimate $ |f(y,v,t)| \leq  C_a e^{-\frac{d}{a}(|(y,v)|^2+ |t|)} $ allows us to extend $ \lambda \rightarrow f^\lambda $ as an $ L^2(\R^{2n}) $ valued holomorphic function on the strip $ |\Im(\lambda)| < d/a.$ Thus, we only need to show that $ f^\lambda = 0 $ for all $ 0 < \lambda < \delta $ for some $ \delta >0.$ Consider the integral
$$ \int_{\R^{2n}} |f^\lambda(y,v)| F_\lambda(y,v) dy dv \leq C_a \int_{\R^{2n}} e^{-\frac{1}{4a} (|y|^2+|v|^2)}   \big(p_{2b^\prime}^\lambda(2iy,2iv) \big)^{1/2}   dy dv.$$
Recalling the expression for $ p_{2b^\prime}^\lambda(2iy,2iv) $ we are led to check the finiteness of the integral
$$ \int_{\R^{2n}} e^{-\frac{1}{4a} (|y|^2+|v|^2)}  e^{\frac{1}{2} \lambda (\coth 2b^\prime \lambda) (|y|^2+|v|^2)}  dy dv.$$
As $  \lambda (\coth 2b^\prime \lambda) $ tends to $ (2 b^\prime)^{-1} < (2 a)^{-1 } $ as $ \lambda $ goes to $ 0$ we can choose $ \delta > 0 $ such that $  \lambda (\coth 2b^\prime \lambda) < (2b)^{-1} $ for $ 0 < \lambda < \delta.$ For such $ \delta $ the above integral is finite.
Therefore we can appeal to Beurling's theorem to complete the proof.\\

In the case of Cowling-Price theorem we have  $  f^\lambda(y,v) = g_\lambda \ast_\lambda p_b^\lambda(y,v) $ and hence  we can take $ F_\lambda(y,v)^2 = p_{2b}^\lambda(2iy,2iv) .$  As before,  the hypothesis $ f = h\, p_a $ allows us to extend $ f^\lambda $  holomorphically to $ |\Im(\lambda)| < d/a.$  Moreover,
$$ |f^\lambda(y,v)| \leq \int_{-\infty}^\infty |h(y,v,t)| p_a(y,v,t) dt  \leq \|h(y,v,\cdot)\|_2  \Big( \int_{-\infty}^\infty (p_a^\lambda(y,v))^2 d\lambda \Big)^{1/2}.$$
From the  explicit expression for $ p_a^\lambda(y,v) ,$ using the fact that $ t \cosh t \geq \sinh t $ for $ t \geq 0,$ we get
$$ \int_{-\infty}^\infty (p_a^\lambda(y,v))^2 d\lambda \leq C  e^{-\frac{1}{2a} (|y|^2+|v|^2)}\,  \int_{-\infty}^\infty \Big( \frac{\lambda}{\sinh a\lambda} \Big)^{2n} d\lambda .$$
Thus we are led to consider the finiteness of the integral
$$ \int_{\R^{2n}} \|h(y,v,\cdot)\|_2 \, e^{-\frac{1}{4a} (|y|^2+|v|^2)}  e^{\frac{1}{2} \lambda (\coth 2b \lambda) (|y|^2+|v|^2)}  dy dv.$$
As before, the above integral is finite for small enough $ \lambda $ which completes the proof of Theorem \ref{C-P}.

\subsection{Beurling's theorem for the Hermite expansions.} We begin with some preparations. Recall that the Hermite projections $ P_kf $ are given by
$$ P_kf(x) = \sum_{|\alpha|=k}  (f, \Phi_\alpha) \, \Phi_\alpha(x) = \int_{\R^n}  f(y) \Phi_k(x,y) dy.$$
It is clear that the kernel $ \Phi_k(x,y) = \sum_{|\alpha|=k}   \Phi_\alpha(x) \, \Phi_\alpha(y) $ extends to $ \C^n \times \C^n $ holomorphically and  $ P_kf(x)$  to $ \C^n.$  Since $ P_k^2 = P_k $ we have
$$  |P_kf(z)| = | \int_{\R^n} P_kf(x^\prime) \Phi_k(x^\prime, z) dx^\prime| \leq ||P_kf||_2  \| \Phi_k(\cdot, z)\|_2.$$ In order to estimate $\| \Phi_k(\cdot, z)\|_2$ we observe that
$$ \| \Phi_k(\cdot, z)\|_2^2 = \sum_{|\alpha|=k} \Phi_\alpha(z) \Phi_\alpha(\bar{z})  =\Phi_k(z,\bar{z}).$$
We now recall the following explicit formula for $ \Phi_k(z,w) $ proved in \cite{T4}, Lemma 2.3.
$$ \Phi_k(z,w) = \pi^{-n/2} \sum_{j=0}^k (-1)^j L_j^{n/2-1}(\frac{1}{2}(z+w)^2)L_{k-j}^{n/2-1}(\frac{1}{2}(z-w)^2) e^{-\frac{1}{2}(z^2+w^2)}$$
so that
$ \Phi_k(z,\bar{z}) = \pi^{-n/2} \sum_{j=0}^k (-1)^j L_j^{n/2-1}(2|x|^2) e^{-|x|^2} L_{k-j}^{n/2-1}(-2|y|^2) e^{|y|^2}.$ We now make use of the estimate (see 1.1.39 in \cite{T0})
$$ |L_j^{n/2-1}(2|x|^2)| e^{-|x|^2} \leq L_j^{n/2-1}(0) =  \frac{\Gamma(j+n/2)}{\Gamma(j+1)\Gamma(n/2)} $$ 
and the formula (see 1.1.51 in \cite{T0})
$$  \sum_{j=0}^k   \frac{\Gamma(j+n/2)}{\Gamma(j+1)\Gamma(n/2)} L_{k-j}^{n/2-1}(|x|^2) e^{-\frac{1}{2}|x|^2} = L_k^{n-1}(|x|^2) e^{-\frac{1}{2}|x|^2} $$ 
in estimating $ \Phi_k(z,\bar{z}).$ We have thus proved
$ \Phi_k(z,\bar{z}) \leq \pi^{-n/2} \varphi_k^{n-1}(2iy) $ and hence
\begin{equation}\label{esti-proj}  |P_kf(z)| \leq ||P_kf||_2  \| \Phi_k(\cdot, z)\|_2 \leq \pi^{-n/2}    ||P_kf||_2 \, \sqrt{\varphi_k^{n-1}(2iy)}.
\end{equation}
With these preparations, we can now launch a proof of Theorem \ref{beur-her}.\\

As usual, we define the function $ F $ on the strip $ S $ by
$$ F(\zeta) = \int_{\R^n}  f(\zeta y)  \overline{f(y)}dy =  \int_{\R^n}  \big( \sum_{k=0}^\infty P_kf(\zeta y) \big)\,  \overline{f(y)}  \, dy .$$
In view of the estimate (\ref{esti-proj}) and the fact that $ \varphi_k^{n-1}(2iy) $ is an increasing radial function we see that under the given assumption on $ f $ 
$$ |F(\zeta)|  \leq C \sum_{k=0}^\infty \int_{\R^n}  |f(y)| \,\|P_kf\|_2\, \sqrt{\varphi_k^{n-1}(2iy)} \, dy <\infty.$$
Thus $ F(\zeta) $ is holomorphic on the strip $ S $ and the rest of the proof goes as before.\\

In order to deduce Hardy's theorem from Theorem \ref{beur-her} we need to check if Beurling's condition is verified. Given $ a \tanh b>1, $ we can choose $ 0 < b^\prime < b $ such that $ a \tanh b^\prime >1 $ still holds. By applying Cauchy-Schwarz and using 
the assumption $\| P_kf\|_2 \leq  C \, e^{-\frac{1}{2}b (2k+n)} $ we get
$$ \sum_{k=0}^\infty  \|P_kf\|_2 \, \big(\varphi_k^{n-1}(2iy)\big)^{1/2}  \leq C_b \, \Big( \sum_{k=0}^\infty  e^{-b^\prime (2k+n)} \varphi_k^{n-1}(2iy) \Big)^{1/2}.$$
Since $ \varphi_k^{n-1}(2iy) = \varphi_{k,1}^{n-1}(2iy,0) ,$ the above sum is nothing but $ p_{b^\prime}^1(2iy,0) $ which is known explicitly. Thus we see that
$$ \sum_{k=0}^\infty  \|P_kf\|_2 \, \big(\varphi_k^{n-1}(2iy)\big)^{1/2} \leq C_b^\prime  \, e^{\frac{1}{2} (\coth b^\prime)|y|^2} .$$
Consequently, as $ |f(y)| \leq c\, e^{-\frac{1}{2} a|y|^2} $ and $ a \tanh b^\prime >1, $  we obtain
$$ \sum_{k=0}^\infty \int_{\R^n}  |f(y)| \,\|P_kf\|_2\, \sqrt{\varphi_k^{n-1}(2iy)} \, dy \leq C\, \int_{\R^n} e^{-\frac{1}{2} (a- \coth  b^\prime) |y|^2 } \, dy  <\infty.$$
Thus Beurling's condition is verified and Hardy's theorem is proved.

\section*{Acknowledgments}   This work  is supported by  J. C. Bose Fellowship from the Department of Science and Technology, Government of India.

\end{document}